\documentclass[a4paper,11pt]{amsart}
\usepackage{amsmath,amssymb,amsthm,amscd,xypic}

\title{The dual of Brown representability for some derived categories}

\author{George Ciprian Modoi}
\address{Babe\c s--Bolyai University, Faculty of Mathematics and Computer Science \\
1, Mihail Kog\u alniceanu, 400084 Cluj--Napoca, Romania}
\email{cmodoi@math.ubbcluj.ro}

\thanks{Research supported by CNCS-UEFISCDI grant PN-II-RU-TE-2011-3-0065}


\date{\today}

\newcommand{\la}{\longrightarrow}

\newcommand{\N}{\mathbb{N}}
\newcommand{\Z}{\mathbb{Z}}

\newcommand{\Prime}{\mathbb{P}}

\DeclareMathOperator{\Ext}{Ext}

\DeclareMathOperator{\ima}{im}

\DeclareMathOperator{\holim}{\underleftarrow{\textrm{\rm holim}}}

\DeclareMathOperator{\Zy}{Z}
\DeclareMathOperator{\Ho}{H}
\DeclareMathOperator{\Bd}{B}

\newcommand{\A}{\mathcal{A}}
\newcommand{\B}{\mathcal{B}}

\newcommand{\CS}{\mathcal{S}}
\newcommand{\T}{\mathcal{T}}

\newcommand{\ModR}{\hbox{\rm Mod-}R}

\newcommand{\Ab}{\mathcal{A}b}

\newcommand{\opp}{^\textit{o}}

\newcommand{\Prod}[1]{\mathrm{Prod}({#1})}
\newcommand{\Prd}{\mathrm{Prod}}
\newcommand{\Der}[1]{\mathbf{D}({#1})}
\newcommand{\Htp}[1]{\mathbf{K}({#1})}
\newcommand{\Htpi}[1]{\mathbf{K}_i({#1})}
\newcommand{\Htpp}[1]{\mathbf{K}_p({#1})}
\newcommand{\Cpx}[1]{\mathbf{C}({#1})}

\newcommand{\Cot}{\mathrm{Cot}}

\theoremstyle{plain}
\newtheorem{thm}{Theorem}[section]
\newtheorem{lem}[thm]{Lemma}
\newtheorem{prop}[thm]{Proposition}
\newtheorem{cor}[thm]{Corollary}

\theoremstyle{definition}

\theoremstyle{remark}
\newtheorem{rem}[thm]{Remark}

\newenvironment{pf}{\noindent {\it Proof of}}

\begin{document}

\begin{abstract}
Consider a complete abelian category which has an injective cogenerator. 
If its derived category is left--complete we show that the dual of this derived category satisfies 
Brown representability. In particular, this is true for the derived category of an abelian AB$4^*$-$n$ category and
for the derived category of quasi--coherent sheaves over a nice enough scheme, including
the projective finitely dimensional space. 
\end{abstract}

\maketitle

\section*{Introduction}

The relevance of derived categories in algebraic geometry has been understood since the time of Grothendieck and his school.  
When one works with derived 
categories, an important problem is to construct adjoints. The main formal tool used for doing this is 
the celebrated Brown Representablity Theorem. This theorem is formulated in a general abstract 
setting, namely for a triangulated category with coproducts, by Neeman in \cite{N}. A 
main problem which remained open in Neeman's book is whether the dual of a 
well--generated triangulated category satisfies Brown representability. In order to fix our notions, let us say 
that Brown representability holds for a triangulated category with coproducts if every cohomological functor which sends 
coproducts to products is representable (by a contravariant hom--functor). For the dual statement, the triangulated category which we work in 
should have products and every homological product preserving functor has to be representable (by a covariant hom--functor). 
In this paper, we show that Brown representability is satisfied by the dual of some well--generated triangulated categories. 
Note that the derived category of quasi--coherent sheaves over a nice enough scheme (including the projective space of finite dimension 
over a commutative ring with one)
fulfills our hypotheses, hence its dual must satisfy Brown representability. 

This paper continues the work from \cite{MS}, \cite{MD}, \cite{MP}, \cite{MK} and \cite{BM}. First,
in \cite{MP} we generalize the idea of \cite{NR} by showing that Brown representability hold for triangulated categories 
with coproducts which are deconstructible in the sense of Lemma \ref{cof} below.  
Next in \cite{MS} it was observed that if the homotopy category of complexes $\Htp\A$ 
over an additive category $\A$ satisfies Brown representability,
then every object in $\A$ must be a direct factor of an arbitrary 
direct product of a fixed object.
To prove the converse, in the paper \cite{MD} it is dualized the approach in \cite{MP}. 
One of the main results in \cite{MD} says that a 
triangulated 
category $\T$ with products satisfies the dual of Brown representability, 
provided that there is a set of objects 
$\CS$, such that every 
object in $\T$ is $\CS$-cofiltered, that is it can be written as a homotopy 
limit of an inverse tower 
with the property that the mapping cone of all connecting 
morphisms are direct factors of direct products of objects in $\CS$. This result is applied in \cite{MK} to the 
homotopy category of projective modules over a ring.
Here we show that, under suitable hypotheses on an abelian category $\A$, every 
complex has a homotopically injective resolution, 
therefore the derived category of $\A$ is equivalent to the category of these 
homotopically injective complexes.
Moreover, these homotopically injective resolutions are constructed in such a way that the mapping 
cone of all connecting morphisms have vanishing differentials, 
allowing us to deduce 
Brown representability for the dual category $\Der\A\opp$. 

Remark that examples provided in this article do not follow from the previous known results in the literature. 
Indeed, neither \cite[Theorem B]{KB} nor \cite[Theorem 8.6.1]{N} do not directly apply, because the categories 
$\Der\A$ which occur in Corollaries \ref{brt4ab4n}, \ref{brt4fid} and 
\ref{brt4quasi} are not necessary compactly generated, and it is not clear whether there is a regular cardinal 
$\alpha$ such that the category of $\alpha$-exact contravariant functors $(\T^\alpha)\opp\to\Ab$ has enough injectives. 
One of the referees suggested that it would be an interesting problem to clarify this last point. Another interesting open 
problem is to clarify the relationship between Brown representability for $\T$ and $\T\opp$.  Although they seem to be independent, 
to the best of our knowledge, there is no example for fixing that fact. On the other hand, under some appropriate set theoretic axioms, 
in \cite{B} it is shown that, for any ring $R$, Brown representability for $\Htp\ModR$ and $\Htp\ModR\opp$ are equivalent.

\section{The main results}

We start by recalling some classical facts and notations concerning derived categories, which are 
necessary in order to formulate our main result. 
Let $\A$ be an abelian category. 
Then its derived category $\Der\A$ is constructed in three steps as follows: 
First, we consider the 
category $\Cpx\A$ of complexes with entries in $\A$, whose objects are 
diagrams of the form
\[X^\bullet=(\cdots\to X^{n-1}\stackrel{d^{n-1}}\la X^n\stackrel{d^n}\la X^{n+1}\to\cdots)^t\]
in $\A$ where $n\in\Z$ and $d^nd^{n-1}=0$. The superscript $()^t$ means the transpose, that is we view a complex as a column, 
for reasons which will be explained later. As usual, we call  
the maps $d_X^n=d^n$ {\em differentials}; we remove the subscript whenever it is 
clear from the context.  
Morphisms in $\Cpx\A$ are collections of maps 
$f^\bullet=(f^n)_{n\in\Z}$ in $\A$ commuting with differentials. 
For a complex $X^\bullet\in\Cpx\A$ and an $n\in\Z$, denote 
$\Zy^n(X^\bullet)=\ker d^n$ and $\Bd^n(X^\bullet)=\ima d^{n-1}$, and call them the object of {\em$n$-th 
cocycles} and the object of $n$-th {\em boundaries}  respectively. 
It is clear that $\Bd^n(X^\bullet)\leq\Zy^n(X^\bullet)\leq X^n$, thus we are allowed to consider 
$\Ho^n(X^\bullet)=\Zy^n(X^\bullet)/\Bd^n(X^\bullet)$, the {\em$n$-th 
cohomology} of $X^\bullet$. A complex $X^\bullet$ is called {\em acyclic} if $\Ho^n(X^\bullet)=0$ 
for all $n\in\Z$.

In the second  
step we construct the homotopy category of complexes over $\A$ as a quotient of $\Cpx\A$. Namely, 
$\Htp\A$ has the same objects as $\Cpx\A$ and 
$\Htp\A(X^\bullet,Y^\bullet)=\Cpx\A(X^\bullet,Y^\bullet)/\sim$, 
where $\sim$ is an equivalence relation called 
{\em homotopy}, defined as follows: Two maps of complexes 
$f^\bullet,g^\bullet:X^\bullet\to Y^\bullet$ are homotopically
equivalent if there are $s^n:X^n\to Y^{n-1}$, for all
$n\in\Z$ such that $f^n-g^n=d_Y^{n-1}s^n+s^{n+1}d_X^n$. This category is 
triangulated, see \cite[Theorem 10.2.4]{W}. The suspension functor, denoted by $[1]$, is an automorphism of $\Cpx\A$ or $\Htp\A$
and it is defined as follows: $X^\bullet[1]^n=X^{n+1}$, $d_{X[1]}^n=-d_X^{n+1}$ and $f^\bullet[1]=(f^{n+1})_{n\in\Z}$. 

Note that two complexes $X^\bullet$ and $Y^\bullet$ are isomorphic 
in $\Htp\A$ if there are maps of complexes $f^\bullet:X^\bullet\to Y^\bullet$ and 
$g^\bullet:Y^\bullet\to X^\bullet$ such that 
both compositions $g^\bullet f^\bullet$ and $f^\bullet g^\bullet$ 
are homotopically equivalent to the respective identities. 
If this is the case, it is 
not hard to see that $X^\bullet$ and $Y^\bullet$ have the same cohomology, so the functors 
$\Ho^n:\Cpx\A\to\A$, $n\in\Z$, induce well defined functors $\Htp\A\to\A$. 
Therefore the full subcategory of acyclic complexes is a 
triangulated subcategory of $\Htp\A$. 
The derived category 
$\Der\A$ is obtained as the Verdier quotient (see \cite[Section 2.1]{N}) 
of $\Htp\A$ modulo the triangulated subcategory of all acyclic complexes. A map  
$f^\bullet:X^\bullet\to Y^\bullet$ in $\Htp\A$ which 
induces isomorphisms in cohomology is called {\em quasi--isomorphism}. Then $\Der\A$ is the
category of fractions of $\Htp\A$ with respect to all quasi--isomorphisms. A priori 
there is no reason to expect that $\Der\A$ has small hom--sets. Note that all categories we consider 
in this paper have small hom--sets, with the unique possible exception of a (Verdier) quotient. 
The statement ``$\Der\A$ has small hom--sets'' says precisely that $\Der\A$ lives in the 
universe we work in.

We shall see every object of $\A$ as a complex concentrated in degree zero, 
providing embeddings of $\A$ in any of the categories $\Cpx\A$, $\Htp\A$
or $\Der\A$. 
Note also that, if $\A$ has (co)products 
then $\Cpx\A$ and $\Htp\A$ have (co)products and the canonical functor $\Cpx\A\to\Htp\A$ preserves them.
If, in addition, these (co)products are exact
then the full subcategory of acyclic complexes is closed under (co)products, therefore $\Der\A$ has also
(co)products and the quotient functor $\Htp\A\to\Der\A$ preserves them, by \cite[Theorem 3.5.1]{KL}. Note that, 
the exactness of (co)products in $\A$ is only a sufficient condition, and not a necessary one, for the existence of (co)products in $\Der\A$.

Let $\T$ be a triangulated category with products, and denote by $[1]$ its suspension functor. 
Recall that if \[X_0\leftarrow X_1\leftarrow X_2\leftarrow \cdots \] is an inverse tower 
(indexed over $\N$) of objects in $\T$, then its 
homotopy limit is defined (up to a non--canonical isomorphism) by the triangle 
\[ \holim X_n\la\prod_{n\in\N}X_n\stackrel{1-shift}\la\prod_{n\in\N}X_n\to\holim X_n[1], \]
see \cite[dual of Definition 1.6.4]{N}.

Now consider  $\T=\Der\A$, where $\A$ is an abelian category. For a complex 
$X^\bullet$ and a positive integer $n\in\N$, consider the truncation \[X^{\geq -n}=(0\to\Bd^{-n}(X^\bullet)\to X^{-n}\to X^{-n+1}\to\cdots)^t.\] 
There is a map of complexes $X^{\geq-(n+1)}\to X^{\geq-n}$ which is the identity $X^i\to X^i$ in  degrees $i\geq-n$, 
the zero map in degrees $i<-(n+1)$ and the canonical epimorphism 
$X^{-(n+1)}\to\Bd^{-n}(X^\bullet)$ in degree $-(n+1)$. In this way, we obtain an inverse tower 
\[X^{\geq0}\leftarrow X^{\geq-1}\leftarrow X^{\geq-2}\leftarrow\cdots.\]  
Then $\Der\A$ it is called {\em left--complete} (see \cite{NN}), provided that it has products and 
$X^\bullet\cong\holim X^{\geq -n}$.   An example of a 
non--left--complete derived category may be found in \cite{NN}. In counterpart,
some examples of left--complete will be provided later.

Let $\T$ be a triangulated category and let $\A$ be an abelian category. 
We call a covariant functor 
$F:\T\to\A$ {\em homological} if it sends triangles into long 
exact sequences. Dually a contravariant functor $F:\T\to\A$ which sends triangles into long exact sequences is called 
{\em cohomological}. Denote by $\Ab$ the category of abelian groups. 
Following \cite{NR}, we say that 
$\T$ satisfies Brown representability, if it has coproducts and  every cohomological functor $F:\T\to\Ab$ which sends 
coproducts into products is representable, that is of the form $F\cong\T(-,X)$ for some $X\in\T$. Dually $\T\opp$ satisfies 
Brown representability if $\T$ has products and every homological 
product preserving functor $F:\T\to\Ab$ is of the form $F\cong\T(X,-)$ for some $X\in\T$.   
Recall that an injective cogenerator for $\A$ is an object $Q\in\A$ such that there is a monomorphism from every other object to a
direct product of copies of $Q$, see  \cite[Chapter IV, \S 6]{S}.

In the sequel we shall formulate our main results:

\begin{thm}\label{brt4der}
 Let $\A$ be a complete abelian category possessing an injective 
cogenerator, and let $\Der\A$ be its derived category. If $\Der\A$ is left--complete, then
$\Der\A$ has small hom--sets and $\Der\A\opp$ satisfies Brown representability.
\end{thm}

Before we prove Theorem \ref{brt4der} we state some immediate consequences. 
Recall that a complete abelian category $\A$ is said to be {\em AB$4^*$-$n$}, with $n\in\N$, 
if the $i$-th derived functor 
of the direct product functor is zero, for all $i>n$ (see also \cite{R} or \cite{HX}). Clearly AB$4^*$-$0$ categories 
are the same as AB$4^*$ categories, that is abelian categories with exact products. 

\begin{cor}\label{brt4ab4n}
  Let $\A$ be an abelian complete category possessing an injective 
cogenerator. If $\A$ is AB$4^*$-$n$, for some $n\in\N$ and $\Der\A$ has products, then
$\Der\A$ has small hom--sets and $\Der\A\opp$ satisfies Brown representability.
\end{cor}

\begin{proof}
 We know by \cite[Theorem 1.3]{HX}, that $\Der\A$ is left--complete, 
 hence Theorem \ref{brt4der} applies.
\end{proof}

Let $\A$ be an abelian category with enough injectives. An {\em injective resolution} of 
$X\in\A$ is a complex of injectives $E^\bullet$ which is zero in 
negative degrees, together with an augmentation map $X\to E^\bullet$, such that the complex 
$0\to X\to E^0\to E^1\to\cdots$ is acyclic.
The {\em injective dimension} of an object
$X\in\A$ is defined to be the smallest $n\in\N$ for which $X$ has an injective resolution of the form
\[0\to X\to E^0\to E^1\to\cdots\to E^{n-1}\to E^n\to 0,\] or $\infty$ if such an injective resolution 
does not exist. Equivalently, $X$ 
has injective dimension $n$ if it is the smallest non--negative integer for which  $\Ext^{n+1}(-,X)=0$. The {\em global injective dimension} of $\A$ is defined to be 
the supremum of all injective dimensions of its objects.

\begin{rem}
 Products in module categories are exact, that is $\ModR$ is AB$4^*$ for every ring $R$ (with or
 without one), hence Corollary \ref{brt4ab4n} applies. But in this case the derived category is known to be compactly generated, 
 hence both $\Der\A$ and $\Der\A\opp$ satisfy Brown representability, for example 
 by \cite[Theorem A and Theorem B]{KB}. An example of a Grothendieck
 AB$4^*$ category which has no nonzero projectives, hence it is not equivalent to a module category, 
 may be found in \cite[Section 4]{R}. Note also that in 
 \cite[Theorem 1.1]{HX} there are other examples of abelian categories $\A$ which are AB$4^*$-$n$, 
 for some $n\in\N$, that is categories for which $\Der\A\opp$ satisfies Brown representability, 
 by Corollary \ref{brt4ab4n} above.   
\end{rem}

\begin{cor}\label{brt4fid}
  Let $\A$ be an abelian complete category possessing an injective 
cogenerator. If $\A$ is of finite global 
injective dimension and $\Der\A$ has products, then 
$\Der\A$ has small hom--sets and $\Der\A\opp$ satisfies Brown representability.
\end{cor}

\begin{proof} We want to apply Corollary \ref{brt4ab4n}, so we will to show that $\A$ is AB$4^*$-$n$, where $n$ is 
the global injective dimension of $\A$. Fix an index set $I$. The $k$-th derived functor of the product $\prod^{(k)}:\A^I\to\A$ can be computed as follows: 
Consider arbitrary objects $X_i\in\A$ with $i\in I$. For every  $i$ choose an injective resolution $X_i\to E_i^\bullet$ of length less than or equal to $n$. 
 Then $\prod^{(k)}X_i=\Ho^k(\prod E_i^\bullet)$, therefore $\prod^{(k)}X_i=0$ for $k>n$.
\end{proof}

\begin{cor}\label{brt4quasi}
 If $\A$ is the category of quasi--coherent sheaves 
 over a quasi--compact and separated scheme then $\Der\A$ has small hom--sets and $\Der\A\opp$ satisfies 
 Brown representability. In particular, if $\Prime^d_R$ is the projective $d$-space, $d\in\N^*$, 
 over an arbitrary commutative ring with one $R$ and $\A$ is the category of quasi--coherent sheaves
 over $\Prime^d_R$, then $\Der\A$ has small hom--sets and $\Der\A\opp$ satisfies 
 Brown representability.
\end{cor}

\begin{proof} The category of quasi--coherent sheaves is Grothendieck, hence  
$\Der\A$ satisfies Brown representability (see for example \cite[Theorem 5.8]{ALS}). 
Consequently, $\Der\A$ has products, by \cite[Proposition 8.4.6]{N}. 
 Moreover, according to \cite[Remark 3.3]{HX}, the category of quasi--coherent sheaves
 over a quasi--compact, separated scheme is AB$4^*$-$n$, for some $n\in\N$.  
 
 Finally,
 $\Prime^d_R$ is obtained by glueing together $d+1$ affine open sets (see \cite[4.4.9]{FAG}).
 Hence, it is quasi--compact (see also exercise \cite[5.1.D]{FAG}). Moreover, $\Prime^d_R$ is separated by 
 \cite[Proposition 10.1.5]{FAG}. 
\end{proof}
 
In the following Corollary we point out that homotopically injective resolutions exist in $\Htp\A$, provided that the
abelian category $\A$ satisfies the hypothesis of Theorem \ref{brt4der}. For technical reasons its proof is postponed 
after the proof of Theorem \ref{brt4der}. 
 
\begin{cor}\label{alt}
The following statements hold for a complete abelian category $\A$ possessing an injective 
cogenerator for which the derived category $\Der\A$ is left--complete:
\begin{itemize}
\item[{\rm (1)}] Every object in $\Htp\A$ has a homotopically injective resolution. 
\item[{\rm (2)}] There is an equivalence of categories $\Htpi\A\stackrel{\sim}\la\Der\A$.
\item[{\rm (3)}] Every additive functor $F:\A\to\B$ to another abelian category $\B$ has a total 
right derived functor ${\mathbf R}F:\Der\A\to\Der\B$
(for details see \cite[1.4]{KP}).
\end{itemize}
\end{cor}

\begin{rem}
Notice that the conclusions of Corollary \ref{alt} are already known for Grothendieck categories 
(see \cite{ALS}). 
Even if the category $\A$ is not necessary Grothendieck, but it satisfies the hypotheses of Theorem \ref{brt4der}, we can easily prove (1), 
but it is not clear if Brown representability for $\Der\A\opp$ can be deduced from this shortest argument.  
\end{rem}

\section{Proof of Theorem \ref{brt4der}}

The first ingredient in the proof of Theorem \ref{brt4der} is a refinement of the technique used by Neeman in \cite{NR} and it is contained in 
\cite{MD}. Here we recall it shortly.
Let $\T$ be a triangulated category with products, and
let $\CS\subseteq\T$ be a set of objects. We denote by $\Prod\CS$ the full
subcategory of $\T$ consisting of all direct factors of products of objects in $\CS$.
We define inductively $\Prd_1(\CS)=\Prod\CS$ and $\Prd_{n+1}(\CS)$ to be
the full subcategory of $\T$ which consists of all objects $Y$
lying in a triangle $X\to Y\to Z\to X[1]$ with $X\in\Prd_1(\CS)$
and $Z\in\Prd_n(\CS)$. Clearly the construction leads to an ascending chain
$\Prd_1(\CS)\subseteq\Prd_2(\CS)\subseteq\cdots$. If we suppose that $\CS=\CS[1]$, then $\Prd_n(\CS)=\Prd_n(\CS)[1]$, 
by \cite[Remark 1.7]{NR}.  The same
\cite[Remark 1.7]{NR} says, in addition, that if $X\to Y\to Z\to X[1]$ is a
triangle with $X\in\Prd_n(\CS)$ and $Z\in\Prd_m(\CS)$ then
$Y\in\Prd_{n+m}(\CS)$. An object $X\in\T$ will be called
{\em$\CS$-cofiltered} if it may be written as a homotopy limit 
$X\cong\holim X_n$ of an inverse tower, with $X_0=0$, and
$X_{n+1}$ lying in a triangle $P_{n}\to X_{n+1}\to X_n\to P_n[1],$
for some $P_n\in\Prod\CS$. Inductively, we have
$X_n\in\Prd_n(\CS)$, for all $n\in\N^*$. Notice that $X$ is $\CS$-cofiltered if and only if $X$ is in 
$\Prd_\omega(\CS)*\Prd_\omega(\CS)$ in the sense of \cite{NR}.
The dual notion is called {\em filtered}. The terminology comes 
from the analogy with the filtered objects in an abelian category (see \cite[Definition 3.1.1]{GT}). 
Using further the same analogy, we say that $\T$ (respectively $\T\opp$) is {\em deconstructible} if 
there is a set (and not a proper class) of objects $\CS=\CS[1]$,  
such that every object $X\in\T$ is $\CS$--filtered (cofiltered). Note that we may define 
deconstructibility without closure under shifts. Indeed, if every $X\in\T$ is 
$\CS$--(co)filtered, then it is also $\overline{\CS}$--(co)filtered, where  $\overline{\CS}$ is the closure of $\CS$ 
under all shifts.

\begin{lem}\label{cof}\cite[Theorem 8]{MD}
If $\T\opp$ is deconstructible, then $\T\opp$ satisfies Brown representability.
\end{lem}

The second ingredient of our proof is an adaptation 
of the argument in \cite[Appendix]{KP}. Fix a 
complete abelian category $\A$, which has an injective cogenerator. 

Recall that a complex $X^\bullet\in\Htp\A$ is called {\em homotopically injective} if 
$\Htp\A(N^\bullet,X^\bullet)=0$, 
for any acyclic complex $N^\bullet$. Denote by $\Htpi\A$ the full subcategory of $\Htp\A$ consisting of 
homotopically injective complexes. It follows immediately, that $\Htpi\A$ is a triangulated subcategory of $\A$ closed under products and direct summands. 
Dually, we can define the {\em homotopically projective} complexes and 
we write $\Htpp\A$ for the 
full subcategory of $\Htp\A$ consisting of such complexes. A {\em homotpically injective resolution}
of a complex $X^\bullet\in\Htp\A$ is by definition a quasi--isomorphism $X^\bullet\to E^\bullet$, with $E^\bullet$ 
homotopically injective. Homotopically injective and projective complexes and resolutions were 
first defined by Spaltenstein in \cite{Sp}, but we follow the approach in \cite{KP}. 
If every complex in $\Htp\A$ has a homotopically injective (projective) resolution, then this resolution yields a left (right) adjoint of the inclusion functor 
$\Htpi\A\to\Htp\A$ (respectively $\Htpp\A\to\Htp\A$); the argument in \cite[1.2]{KP} generalizes with no change in this more general case.
For example, if $R$ is a ring and $\A=\ModR$ is the category of 
all right modules over $R$, then $\A$ has enough projective and enough injective objects, and by \cite[1.1. and 1.2]{KP}  we have equivalences of categories 
\[\Htpp\ModR\stackrel{\sim}\la\Der\ModR\stackrel{\sim}\longleftarrow\Htpi\ModR.\]
More generally, if $\A$ is a Grothendieck category, 
it may not have enough projectives, and the left--side functor might not be an equivalence.  
But it must have enough injectives, and the right--side equivalence must hold as it can be seen 
from \cite[Section 5]{ALS}. Another proof of this fact is contained in \cite[Section 3]{F}.

We consider double complexes with entries in $\A$, 
whose differentials go from 
bottom to top and from left to right. That is, a double complex is a commutative diagram of the form:
\[X^{\bullet,\bullet}=\left(\diagram 
X^{i+1,j}\rto^{d^{i+1,j}_h}            & X^{i+1,j+1} \\
X^{i,j}\uto^{d^{i,j}_v}\rto_{d^{i,j}_h} & X^{i,j+1}\uto_{d^{i,j+1}_v}
\enddiagram\right)_{i,j\in\Z}\]
such that $d_v^2=0=d_h^2$. We denote by $X^{\bullet,j}$ the columns and by $X^{i,\bullet}$  
the rows of $X^{\bullet,\bullet}$.

Let $X^\bullet\in\Cpx\A$ be a complex. We identify it with a double complex concentrated in the 0-th column, making explicit the reason for which simple complexes are columns.
A {\em Cartan--Eilenberg injective resolution} for
$X^\bullet$ ({\em CE injective resolution} for short) is a right half--plane double complex $E^{\bullet,\bullet}$ 
(that is $E^{i,j}=0$ for $j<0$), together with an augumentation map  (of double complexes) 
$X^\bullet\to E^{\bullet,\bullet}$ (with the identification above) such that 
$E^{i,\bullet}=0$ provided that $X^i=0$ and 
 the induced sequences  
\[0\to\Ho^i(X^\bullet)\to\Ho^i(E^{\bullet,0})\to\Ho^i(E^{\bullet,1})\to\cdots,\] 
\[0\to\Bd^i(X^\bullet)\to\Bd^i(E^{\bullet,0})\to\Bd^i(E^{\bullet,1})\to\cdots\]
are injective resolutions for all $i\in\Z$ (see \cite[Definition 5.7.1]{W}). If $X^\bullet\to E^{\bullet,\bullet}$ is a CE injective resolution, then the induced sequences 
\[0\to\Zy^i(X^\bullet)\to\Zy^i(E^{\bullet,0})\to\Zy^i(E^{\bullet,1})\to\cdots,\] 
\[0\to X^i\to E^{i,0}\to E^{i,1}\to\cdots\]
are injective resolutions for all $i\in\Z$ (see \cite[Exercise 5.7.1]{W}). For constructing a CE injective resolution 
for a given complex $X^\bullet$ we start with injective resolutions for $\Ho^i(X^\bullet)$ and $\Bd^i(X^\bullet)$, for all $i\in\Z$. 
Since the sequences $0\to\Bd^i(X^\bullet)\to\Zy^i(X^\bullet)\to\Ho^i(X^\bullet)\to0$ and 
$0\to\Zy^{i}(X^\bullet)\to X^i\to\Bd^{i+1}(X^\bullet)\to0$ are short exact, 
we use horseshoe lemma in order to construct injective 
resolutions for $\Zy^i(X^\bullet)$ and $X^i$. Assembling together these data we obtain the desired CE injective resolution is $X^\bullet\to E^{\bullet,\bullet}$ (see also 
\cite[Lemma 5.7.1]{W}). If $E^{\geq-n,\bullet}$ is the truncated double complex having the columns 
\[E^{\geq-n,j}=(0\to\Bd^{-n}(E^{\bullet,j})\to E^{-n.j}\to E^{-n+1,j}\to\cdots)^t\]
then by the very definition of a CE injective resolution we infer that $X^{\geq-n}\to E^{\geq-n,\bullet}$ is a CE injective resolution for the truncated complex.

\begin{rem}\label{split}The sequences 
$0\to\Bd^{i}(E^{\bullet,j})\to\Zy^{i}(E^{\bullet,j})\to\Ho^{i}(E^{\bullet,j})\to 0$ and $0\to\Zy^{i}(E^{\bullet,j})\to E^{i,j}\to\Bd^{i+1}(E^{\bullet,j})\to 0$ 
have injective components, hence they are split exact for all $i,j\in\Z, j\geq0$.
\end{rem}

Next we define the {\em cototalization} of a double complex $X^{\bullet,\bullet}$ with $X^{i,j}\in\A$ as the simple complex $\Cot(X^{\bullet,\bullet})$ having entries:
\[\Cot(X^{\bullet,\bullet})^n=\prod_{i+j=n}X^{i,j}\] and whose differentials are induced by using the universal property 
of the product by the maps 
\[\prod_{i+j=n}X^{i,j}\to X^{p,q-1}\times X^{p-1,q}\stackrel{(d_h^{p,q-1},d_v^{p-1,q})}\la X^{p,q}\]
for all $p,q\in\Z$ with $p+q=n+1$.

\begin{lem}\label{ce-hi} Consider a 
complete abelian category $\A$, which has an injective cogenerator. 
 If $X^{\bullet}\to E^{\bullet,\bullet}$ is a CE injective resolution of the complex 
 $X^\bullet\in\Cpx\A$ then 
 $\Cot(E^{\bullet,\bullet})\cong\displaystyle{\holim}\Cot(E^{\geq-n,\bullet})$ is 
 homotopically injective.
\end{lem}

\begin{proof} For every $n\in\N$ we observe that there is a map of double complexes
 $E^{\geq-(n+1),\bullet}\to E^{\geq-n,\bullet}$ which is the identity $E^{i,\bullet}\to E^{i,\bullet}$ for $i\geq -n$, the zero map 
 for $i<-(n+1)$ and the epimorphism $E^{-(n+1),\bullet}\to\Bd^{-n}(E^{\bullet,\bullet})$ for $i=-(n+1)$.  Hence Remark \ref{split} tells us 
 that  $E^{\geq-(n+1),\bullet}\to E^{\geq-n,\bullet}$ are split epimorphisms in each degree, 
 for every $n\in\N$. According to \cite[Lemma 2.17]{MT} they induce degree--wise split epimorphisms 
 \[\Cot(E^{\geq-(n+1),\bullet})\to\Cot(E^{\geq-n,\bullet}),\] for all $n\in\N$. Thus there is a degree--wise 
 split short exact sequence in $\Cpx\A$
 \[0\to\displaystyle{\underleftarrow{\lim}}\Cot(E^{\geq-n,\bullet})\to 
 \prod_{n\in\N}\Cot(E^{\geq-n,\bullet})\stackrel{1-shift}\la\prod_{n\in\N}\Cot(E^{\geq-n,\bullet})
 \to0 \] which induces a triangle in $\Htp\A$. On the other hand, we have
 \[\displaystyle{\underleftarrow{\lim}}\Cot(E^{\geq-n,\bullet})\cong\Cot(E^{\bullet,\bullet})\] in $\Cpx\A$, 
 and the induced triangle leads to an isomorphism \[\displaystyle{\holim}\Cot(E^{\geq-n,\bullet})\cong\Cot(E^{\bullet,\bullet})\] in 
 $\Htp\A$ (see also \cite[Lemma 2.6]{IK}). 
 As we noticed, $\Htpi\A$ is a triangulated subcategory closed 
 under products, hence it is also closed under homotopy limits. Finally it remains to show that $\Cot(E^{\geq-n,\bullet})$ is homotopically
 injective for all $n\in\N$. But this property holds for 
  bounded below complexes having injective entries (see for example \cite[Corollary 10.4.7]{W}), in particular it is true for $\Cot(E^{\geq-n,\bullet})$ too.
\end{proof}

For every complex $X^\bullet\in\Cpx\A$ having a CE injective resolution 
$X^\bullet\to E^{\bullet,\bullet}$ we have an obvious map $X^\bullet\to\Cot(E^{\bullet,\bullet})$.
Sometimes it happens that this map is a quasi--isomorphism,  in which case Lemma \ref{ce-hi} above
tells us that it is a homotopically injective resolution. The following lemma shows that this is always the case for 
bounded below complexes, that is complexes $X^\bullet$ for which  $X^n=0$ for $n<<0$.

\begin{lem}\label{ce-plus} Consider a 
complete abelian category $\A$, which has an injective cogenerator. 
 Let $X^\bullet\in\Cpx\A$ be a bounded below complex and 
 let $X^{\bullet}\to E^{\bullet,\bullet}$ be a CE injective resolution. Then 
 $X^\bullet\to\Cot(E^{\bullet,\bullet})$ is a homotopically injective resolution.  
\end{lem}

\begin{proof}
 Without losing the generality, we may suppose that $X^j=0$ for all $j<0$, so 
 $E^{i,j}=0$ for $i<0$ or $j<0$. Consider the bicomplex 
 \[A^{\bullet,\bullet}=0\to X^\bullet\to E^{\bullet,0}\to E^{\bullet,1}\to\cdots,\]
 that is the bicomplex whose first column is $X^\bullet$ followed by the columns of 
 $E^{\bullet,\bullet}$ shifted by $-1$.  The sequence of bicomplexes $A^{\bullet,\bullet}\to X^\bullet\to E^{\bullet,\bullet}$ 
 induces a  triangle \[\Cot(A^{\bullet,\bullet})\to X^\bullet\to \Cot(E^{\bullet,\bullet})\to
 \Cot(A^{\bullet,\bullet})[1]\] 
 in $\Htp\A$, since $\Cot(A^{\bullet,\bullet})[1]$ is the mapping cone of $X^\bullet\to \Cot(E^{\bullet,\bullet})$. 
 Now $A^{\bullet,\bullet}$ is a first quadrant bicomplex (that is $A^{i,j}=0$ for $i<0$ or $j<0$) with acyclic rows.
 We claim its cototalization is acyclic, and the triangle above 
 proves our lemma. 
 
 Because $A^{\bullet,\bullet}$ lies in the first quadrant, it follows that $\Cot(A^{\bullet,\bullet})^n=0$ for $n<0$. 
 Fix $n\geq 0$, and let $A^{\leq n+1,\bullet}$ be the truncation of $A$ obtained by deleting the rows in degree $>n+1$, and 
 replacing the $(n+1)$-th row with \[\cdots\to\Zy^{n+1}(A^{i,\bullet})\to\Zy^{n+1}(A^{i+1,\bullet})\to\cdots.\] Since, for $0\leq m\leq n+1$, the computation  of 
 $\Cot(A^{\bullet,\bullet})^m$ involves only the rows $A^{i,\bullet}$ with $0\leq i\leq m$, therefore 
 $\Cot(A^{\bullet,\bullet})^k=\Cot(A^{\leq n,\bullet})^k$, for all $0\leq k\leq n$.  But $A^{\leq n,\bullet}$  is a first quadrant 
 bicomplex with acyclic rows which has only finitely many 
 non--zero rows, therefore we can obtain $\Cot(A^{\leq n,\bullet})$ in finitely 
 many steps by forming triangles whose cones are the rows. This shows that $\Cot(A^{\leq n,\bullet})$ is acyclic, hence $\Cot(A^{\bullet,\bullet})$
 is acyclic in degree $n$.  Because $n$ is arbitrary our claim is proved (see also \cite[Lemma 2.19]{MT}). 
\end{proof}

\begin{prop}\label{ce-inj} Consider a 
complete abelian category $\A$, which has an injective cogenerator, such that  
 $\Der\A$ has products. Suppose also that 
 for any complex in $X^\bullet\in\Cpx\A$ the 
 cototalization of any CE injective resolution $X^\bullet\to E^{\bullet,\bullet}$ provides 
 a homotopically injective resolution $X^\bullet\to\Cot(E^{\bullet,\bullet})$. Then $\Der\A$ has 
 small hom--sets, $\Der\A\opp$ is deconstructible and $\Der\A\opp$ satisfies 
 Brown representabily. \end{prop}

\begin{proof} By hypothesis, $X^\bullet\to\Cot(E^{\bullet,\bullet})$ is a homotopically injective resolution, for every $X^\bullet\in\Htp\A$.  Completing it to a triangle 
\[N^\bullet\to X^\bullet\to\Cot(E^{\bullet,\bullet})\to N^\bullet[1]\] we deduce that 
$N^\bullet$ is acyclic, that is $\Htp\A(N^\bullet,I^\bullet)=0$ for all $I^\bullet\in\Htpi\A$. 
By standard arguments concerning Bousfield localizations, see \cite[dual of Theorems 9.1.16 and Theorem 9.1.13]{N},  
 we obtain an equivalence of categories $\Htpi\A\stackrel{\sim}\la\Der\A$, so $\Der\A$ has small hom--stes. 
 
 Note that every complex $X^\bullet$ is isomorphic in $\Der\A$ to $\Cot(E^{\bullet,\bullet})$. 
 Moreover, Lemma \ref{ce-hi} implies 
 $\Cot(E^{\bullet,\bullet})\cong\displaystyle{\holim}\Cot(E^{\geq-n,\bullet})$.
 But, for every $n\in\N$, the kernel of the degree--wise split epimorphism of complexes 
 (see the proof of Lemma \ref{ce-hi}) $\Cot(E^{\geq-(n+1),\bullet})\to\Cot(E^{\geq-n,\bullet})$
 is the complex \begin{align*} 
 0\to\Bd^{-(n+1)}(E^{\bullet,0})&\to\Zy^{-(n+1)}(E^{\bullet,0})\times\Bd^{-(n+1)}(E^{\bullet,1})\\ &\to 
 \Zy^{-(n+1)}(E^{\bullet,1})\times\Bd^{-(n+1)}(E^{\bullet,2})\to\cdots  
 \end{align*}
  with differentials being represented as matrices whose components 
 are the inclusions $\Bd^{-(n+1)}(E^{\bullet,j})\to\Zy^{-(n+1)}(E^{\bullet,j})$ 
 and $0$ otherwise.  Computing the cohomology of this complex we can see that it is 
 quasi-isomorphic, therefore isomorphic in $\Htpi\A$, to the complex: 
 \[0\to\Ho^{-(n+1)}(E^{\bullet,0})\to\Ho^{-(n+1)}(E^{\bullet,1})\to\Ho^{-(n+1)}(E^{\bullet,2})\to\cdots,\]  with vanishing differentials. 
 But this last complex is the product of its subcomplexes concentrated in each degree 
 and all entries are  
 injective, hence they are  
 direct summands of a product of copies of $Q$, where $Q$ is an injective cogenerator of $\A$. 
 Therefore, every object in $\Htpi\A$ is $\CS$-cofiltered, for $\CS=\{Q[n]\mid n\in\Z\}$, and 
 Lemma \ref{cof} applies.
\end{proof}

\begin{pf} {\it Theorem \ref{brt4der}}. We want to apply Proposition \ref{ce-inj}, hence we have to 
show that, if $\Der\A$ is left--complete, then the cototalization of a CE injective resolution 
$X^\bullet\to E^{\bullet,\bullet}$ provides 
 a homotopically injective resolution for the complex $X^\bullet\in\Cpx\A$.
 This is true for the truncated complexes $X^{\geq-n}$ for all $n\in\N$, by Lemma \ref{ce-plus} above,
 since $X^{\geq-n}\to E^{\geq-n,\bullet}$ is also a CE injective resolution.
 Therefore, $X^{\geq-n}\cong\Cot(E^{\geq-n,\bullet})$ in $\Der\A$. Taking homotopy limits 
 and using the hypothesis and Lemma \ref{ce-hi} we obtain:
 \[X\cong\displaystyle{\holim}X^{\geq-n}\cong\displaystyle{\holim}\Cot(E^{\geq-n,\bullet})\cong
 \Cot(E^{\bullet,\bullet})\] and the proof is complete.
 \qed 
\end{pf}

\begin{rem}
 For complexes of $R$-modules, where $R$ is a ring, it is showed in \cite{KP} that the cototalization of a CE injective resolution
 provides a homotopically injective resolution. The technique used there for doing this stresses the so called Mittag--Leffler 
 condition, which says that limits of inverse towers whose connecting maps are surjective are exact. 
 Amnon Neeman pointed out that Mittag--Leffler condition doesn't work in the more 
 general case of Grothendieck categories, as it may be seen from \cite[Corollary 1.6]{R}. Consequently the argument of 
 Keller in \cite{KP} may not be used without changes in the case of Grothendieck categories. The fact detailed in this Remark was learned from Amnon Neeman.
\end{rem}

\begin{pf} {\it Corollary \ref{alt}}.
As we have already seen the hypotheses of Proposition \ref{ce-inj} are satisfied, hence (1) and (2) hold as it is established in the proof of this Proposition. 
From here the statement (3) is straightforward. \qed
\end{pf}

\section*{Acknowledgement} The author is indebted to Amnon Neeman for several comments and 
suggestions regarding 
this work, especially for pointing him that Mittag--Leffler condition doesn't work in the case of Grothendieck categories. Finally, 
he wishes to thank the anonymous referees for many
helpful suggestions which lead to a considerable improvement of this paper.

\end{document}